\documentclass[11pt]{amsart}

\usepackage{amssymb}
\usepackage{mathrsfs}
\usepackage{amsfonts}
\usepackage{latexsym,amsmath,amsthm,amssymb,amsfonts}
\usepackage[usenames]{color}
\usepackage{xspace,colortbl}
\usepackage{graphicx}
\usepackage{tipa}
\usepackage{txfonts}

\newcommand{\be}{\begin{equation}}
\newcommand{\ee}{\end{equation}}
\newcommand{\beq}{\begin{eqnarray}}
\newcommand{\eeq}{\end{eqnarray}}

\newtheorem{prop}{Proposition}[section]

\newtheorem{remark}[prop]{Remark}

\newtheorem{defi}[prop]{Definition}

\newtheorem{pro}[prop]{Problem}

\def\begeq{\begin{equation}}
\def\endeq{\end{equation}}

\def\odot{\setbox0=\hbox{$\bigcirc$}\relax \mathbin {\hbox
to0pt{\raise.5pt\hbox to\wd0{\hfil $\wedge$\hfil}\hss}\box0 }}

\numberwithin{equation} {section}

\numberwithin{equation}{section}
\newtheorem{theorem}{\bf Theorem}[section]

\newtheorem{assumption}[theorem]{\bf Assumption}
\newtheorem{lemma}[theorem]{\bf Lemma}

\newtheorem{corollary}[theorem]{\bf Corollary}

\allowdisplaybreaks

\begin{document}

\title[Dual Orlicz Minkowski problems]
 {Flow by Gauss curvature to Dual Orlicz-Minkowski problems}

\author{
Li Chen, Qiang Tu, Di Wu, Ni Xiang}
\address{
Hubei Key Laboratory of Applied Mathematics, Faculty of Mathematics and Statistics,
Hubei University, Wuhan 430062, China } \email{chernli@163.com, qiangtu@hubu.edu.cn,
wudi19950106@126.com, nixiang@hubu.edu.cn}

\thanks{
This research was supported by Hubei Provincial Department of
Education Key Projects D20171004, D20181003 and the National Natural
Science Foundation of China No.11971157.}

\date{}
\begin{abstract}
In this paper we study a normalised anisotropic Gauss curvature flow
of strictly convex, closed hypersurfaces in the Euclidean space
$\mathbb{R}^{n+1}$. We prove that the flow exists for all time and
converges smoothly to  the unique, strictly convex solution of a
Monge-Amp\`ere type equation. Our argument provides a parabolic
proof in the  smooth category for the existence of solutions to the
Dual Orlicz-Minkowski problem introduced by Zhu, Xing and Ye.
\end{abstract}

\maketitle {\it \small{{\bf Keywords}: Gauss curvature flow, convex
hypersurface, Monge-Amp\`ere equation.}

}

\section{Introduction}

As we known, the Gauss curvature flow was introduced by Firey
\cite{Fi74}  to model the shape change of worn stones.  The first
celebrated result was proved by Andrews in \cite{An99} for Gauss
curvature flow, where Firey's conjecture that convex surfaces moving
by their Gauss curvature become spherical as they contract to points
was proved. Guan and Ni \cite{Guan17} proved that convex
hypersurfaces in $\mathbb{R}^{n+1}$ contracting by the Gauss
curvature flow converge (after rescaling to fixed volume) to a
smooth uniformly convex self-similar solution of the flow. Soon,
Andrews, Guan and Ni \cite{An16} extended the results in
\cite{Guan17} to the flow by powers of the Gauss curvature
$K^{\alpha}$ with $\alpha>\frac{1}{n+2}$. Recently, Brendle, Choi
and Daskalopoulos \cite{Br17} proved that round spheres are the only
closed, strictly convex self-similar solutions to the
$K^{\alpha}$-flow with $\alpha>\frac{1}{n+2}$. Therefore, the
generalized Firey's conjecture proposed by Andrews in \cite{An96}
was completely solved, that is, the solutions of the flow by powers
of the Gauss curvature converge to spheres for any
$\alpha>\frac{1}{n+2}$. We also refer to \cite{Ch85,An98,An03,An12}
and the references therein.

As a natural extension of Gauss curvature flows, anisotropic Gauss
curvature flows have attracted considerable attention and they
provide alternative proofs for the existence of solutions to
elliptic PDEs arising in geometry and physics, especially for the
Minkowski-type problem. For example a alternative proof based on the
logarithmic Gauss curvature flow was given by Chou-Wang in
\cite{Ch00} for the classical Minkowski problem, in \cite{Wang96}
for a prescribing Gauss curvature problem. Using  a contracting
Gauss curvature flow, Li-Sheng-Wang \cite{LiQ} have provided a
parabolic proof  in the smooth category for the classical
Aleksandrov and dual Minkowski problems. Recently, two kinds of
normalised anisotropic Gauss curvature flow are used to prove the
$L_p$ dual Minkowski problems by Chen-Huang-Zhao \cite{Chen-H19} and
Chen-Li \cite{Chen-L19}, respectively. These results are major
source of inspiration for us.

Let $\mathcal{M}_0$ be a smooth, closed, strictly convex
hypersurface in $\mathbb{R}^{n+1}$ enclosing the origin. In this paper, we study
the long-time behavior of the following normalised anisotropic Gauss
curvature flow which is a family of hypersurfaces $\mathcal{M}_t$
given by smooth maps $X: \mathcal{M}\times [0, T)\rightarrow
\mathbb{R}^{n+1}$ satisfying the initial value problem
\begin{equation}\label{N-Eq}
\left\{
\begin{aligned}
&\frac{\partial X}{\partial t}=-\theta(t)f(\nu)
\frac{r^{n+1}}{\varphi(r)}K\nu+X,&\\
&X(\cdot,0)=X_{0},
\end{aligned}
\right.
\end{equation}
where $\nu$ is the unit outer vector of $\mathcal{M}_t$ at $X$, $K$
denotes the Gauss curvature of $\mathcal{M}_t$ at $X$, $r=|X|$
denotes the distance form $X$ to the origin, $f\in
C^{\infty}(\mathbb{S}^n)$ with $f>0$, and
$$\theta(t)=\int_{\mathbb{S}^n}\varphi(r(\xi, t))d\xi \
\bigg[\int_{\mathbb{S}^n}f(x)dx\bigg]^{-1}.$$ Notice that $u$
denotes the support function of $\mathcal{M}_t$ given by $u=\langle
X, \nu\rangle$ and $\varphi$ is a positive smooth function.

The reason that we study the  flow \eqref{N-Eq} is to explore the
existence of the smooth solutions to the dual Orlicz-Minkowski
problem introduced by Zhu-Xing-Ye \cite{Zhu20}, which is related to
the following   Monge-Amp\`ere type equation
\begin{eqnarray}\label{Min}
\frac{u \ \varphi(r)}{r^{n+1}}\cdot \mbox{det} (u_{ij}+ u \
\delta_{ij})=f(x) \quad \mbox{on} \quad \mathbb{S}^n,
\end{eqnarray}
where $r=\sqrt{|Du|^2+u^2}$.
In deed, let $\mathcal{K}_0$ be the set of all convex bodies in $\mathbb{R}^{n+1}$ which contain the origin in their interiors, $\varphi :(0, +\infty)\rightarrow (0, +\infty)$ be a continuous function.
Zhu-Xing-Ye \cite{Zhu20} have introduced the definition of the dual Orlicz curvature measure $\widetilde{C}_{\varphi}(K, \cdot)$,
and posed the following dual Orlicz-Minkowski problem:

\begin{pro}[Dual Orlicz-Minkowski problem]
Under what conditions on $\varphi$ and a nonzero finite Borel measure $\mu$
 on $\mathbb{S}^n$, there exists a constant $c>0$ and a $K\in \mathcal{K}_0$ such that $\mu=c\widetilde{C}_{\varphi}(K, \cdot)$?
 \end{pro}
 When $\mu$ has a density $f$, this Minkowski problem is equivalent to solve the Monge-Amp\`ere type equation \eqref{Min}.
 When $\varphi(r)=r^q$, this becomes the dual Minkowsi problem for the $q$-th dual curvature considered by Huang-Lutwak-Yang-Zhang \cite{Huang16}. It is worth pointing out that  they also proved the existence of symmetric
 solutions for the case $q\in (0, n+1)$ under some conditions.
 For $q=n+1$, the dual Minkowski problem becomes the logarithmic Minkowski problem which studied in \cite{BLYZ}.
 For $q<0$,the existence and uniqueness of weak solution were obtained by Zhao \cite{ZY}.

It is to be expected that the  flow \eqref{N-Eq} converges to the
solution of the equation \eqref{Min}. The main idea is to find a
suitable functional which is monotonic under the flow \eqref{N-Eq}.
The difficulty of our proof lies the inhomogeneous term
$\varphi(r)$. To statement our theorem, we need the following
assumption.

\begin{assumption}\label{Ass1}
$\Phi: (0, +\infty)\rightarrow (0, +\infty)$ is a continuous
function such that
$$\Phi(t)=\int_{0}^{t}\frac{\varphi(s)}{s}ds$$ exists for every
$t>0$.
\end{assumption}

\begin{theorem}\label{main}
Assume that $f \in C^{\infty}(\mathbb{S}^n)$ is a positive smooth
function and $\varphi: (0, +\infty)\rightarrow (0, +\infty)$ is a
smooth function. Let $\mathcal{M}_0 \subset \mathbb{R}^{n+1}$ be a
strictly convex, closed hypersurface which contains the origin in
its interior. \

(i) If $\max_{s>0}s\varphi'(s)\varphi^{-1}(s)<0$ for any $t \in (0, +\infty)$,
then the normalised flow \eqref{N-Eq} has a unique smooth solution,
which exists for any time $t \in [0, \infty)$. For each $t \in [0,
\infty)$, $\mathcal{M}_t=X(\mathbb{S}^n, t)$ is a closed, smooth and
strictly convex hypersurface and the support function $u(x, t)$ of
$\mathcal{M}_t=X(\mathbb{S}^n, t)$ converges smoothly, as
$t\rightarrow \infty$, to the unique positive, smooth and strictly
convex solution of the equation \eqref{Min} with $f$ replaced by
$\lambda_0f$ for some $\lambda_0>0$. \

(ii) Under the assumption \eqref{Ass1}, if $f$ is in addition even
function and the initial hypersurface $\mathcal{M}_0$ is
origin-symmetric, then the normalised flow \eqref{N-Eq} has a unique
smooth solution, which exists for any time $t \in [0, \infty)$. For
each $t \in [0, \infty)$, $\mathcal{M}_t=X(\mathbb{S}^n, t)$ is a
closed, smooth, strictly convex and origin-symmetric hypersurface
and the support function $u(x, t)$ of $\mathcal{M}_t=X(\mathbb{S}^n,
t)$ converges smoothly, as $t\rightarrow \infty$, to the unique
positive, smooth, strictly convex and even solution of the equation
\eqref{Min} with $f$ replaced by $\lambda_0f$ for some
$\lambda_0>0$.
\end{theorem}

\begin{remark}
If $\varphi(r)=r^q$, the assumption
$\max_{s>0}s\varphi'(s)\varphi^{-1}(s)<0$ means $q<0$, and the
assumption \eqref{Ass1} is equivalent to $q\geq 0$, thus Theorem
\ref{main} recovers a parabolic proof in the smooth category for the
existence of solutions to the dual Minkowsi problem which given in
\cite{LiQ}.
\end{remark}

The organization of this paper is as follows. In Sect. 2 we start
with some preliminaries. In Sect. 3 we obtain $C^0$ and $C^1$
estimates. The $C^2$ estimates are given in Sect. 4. In Sect. 5 we
prove Theorem \ref{main}.

\section{Preliminaries}

\subsection{Setting and General facts}

\

For convenience, we first state our conventions on Riemann Curvature tensor and derivative notation.
Let $M$ be a smooth manifold and $g$ be a Riemannian metric on
$M$ with Levi-Civita connection $D$. For a $(s, r)$ tensor field $\alpha$ on $M$, its
covariant derivative $D \alpha$ is a $(s, r+1)$ tensor field given by
\begin{eqnarray*}
&&D \alpha(Y^1, .., Y^s, X_1, ..., X_r, X)
\\&=&D_{X} \alpha(Y^1, .., Y^s, X_1, ..., X_r)\\&=&X(\alpha(Y^1, .., Y^s, X_1, ..., X_r))-
\alpha(D_X Y^1, .., Y^s, X_1, ..., X_r)\\&&-...-\alpha(Y^1, .., Y^s, X_1, ..., D_X  X_r),
\end{eqnarray*}
the coordinate expression of which is denoted by
$$D \alpha=(\alpha_{k_{1}\cdot\cdot\cdot
k_{r}; k_{r+1}}^{l_{1}\cdot\cdot\cdot
l_{s}}).$$
We can continue to define the second covariant derivative of $\alpha$ as follows:
\begin{eqnarray*}
&&D^2 \alpha(Y^1, .., Y^s, X_1, ..., X_r, X, Y)
=(D_{Y}(D\alpha))(Y^1, .., Y^s, X_1, ..., X_r, X),
\end{eqnarray*}
the coordinate expression of which is denoted by
$$D^2 \alpha=(\alpha_{k_{1}\cdot\cdot\cdot
k_{r}; k_{r+1}k_{r+2}}^{l_{1}\cdot\cdot\cdot
l_{s}}).$$
Similarly, we can also define the higher order covariant derivative of $\alpha$:
$$D^3 \alpha=D(D^2 \alpha), ... $$
and so on.
For simplicity, the coordinate expression of the covariant differentiation will usually be
denoted by indices without semicolons, e.g. $$u_{i}, \quad u_{ij} \quad \mbox{or} \quad u_{ijk}$$ for
a function $u: M\rightarrow \mathbb{R}$.

Our convention for the Riemannian curvature (3,1)-tensor $Rm$ is defined by
\begin{equation*}
Rm(X, Y)Z=-D_{X}D_{Y}Z+D_{Y}D_{X}Z+D_{[X,
Y]}Z.
\end{equation*}
Pick a local coordinate chart $\{x^i\}_{i=1}^{n}$ of $M$. The
component of the (3,1)-tensor $Rm$ is defined by
\begin{equation*}
Rm\bigg({\frac{\partial}{\partial x^i}}, {\frac{\partial}{\partial x^j}}\bigg){\frac{\partial}{\partial x^k}}
\doteq R_{ijk}^{\ \ \ l}{\frac{\partial}{\partial x^l}}
\end{equation*}
and $R_{ijkl}=g_{lm}R_{ijk}^{\ \ \ m}$. Then, we have the standard commutation formulas (Ricci identities):
\begin{eqnarray}\label{RI}
\alpha_{k_{1}\cdot\cdot\cdot
k_{r};\ j i}^{l_{1}\cdot\cdot\cdot
l_{s}}-\alpha_{k_{1}\cdot\cdot\cdot
k_{r};\ i j}^{l_{1}\cdot\cdot\cdot
l_{s}}=\sum_{a=1}^{r}R^{\ \ \ m}_{ijk_{l}} \alpha_{k_{1}\cdot\cdot\cdot
k_{a-1}m k_{a+1}\cdot\cdot\cdot k_{r}}^{l_{1}\cdot\cdot\cdot
l_{s}}-\sum_{b=1}^{s}R^{\ \ \ l_b}_{ijm} \alpha_{k_{1}\cdot\cdot\cdot
k_{r}}^{l_{1}\cdot\cdot\cdot
l_{b-1}m l_{b+1}\cdot\cdot\cdot l_{r}}.
\end{eqnarray}
We list some facts which will be used frequently.
For the standard sphere $\mathbb{S}^n$ with the sectional curvature $1$,
$$R_{ijkl}=\delta_{ik}\delta_{jl}-\delta_{il}\delta_{jk}.$$
A special case of Ricci identity for a function $u: M \rightarrow \mathbb{R}$
will be usually used frequently:
\begin{equation*}
u_{kji}-u_{kij}=R^{\ \ \ m}_{ijk} u_m.
\end{equation*}
In particular, for a function $u: \mathbb{S}^n\rightarrow \mathbb{R}$,
\begin{eqnarray}\label{3C}
u_{kji}-u_{kij}=\delta_{ik}u_j-\delta_{jk}u_i.
\end{eqnarray}

Let $(M, g)$ be an immersed hypersurface in
$\mathbb{R}^{n+1}$ and $\nu$ be a given unit outward normal. The
second fundamental form $h_{ij}$ of the hypersurface $M$ with respect to
$\nu$ is defined by
\begin{eqnarray*}
h_{ij}=-\left\langle\frac{\partial^2 X}{\partial x^i\partial x^j},
\nu\right\rangle_{\mathbb{R}^{n+1}}.
\end{eqnarray*}

\subsection{Basic properties of convex hypersurfaces}

\

We first recall some basic properties of convex hypersurfaces.
Let $\mathcal{M}$ be a smooth, closed, uniformly convex hypersurface in
$\mathbb{R}^{n+1}$. Assume that $\mathcal{M}$ is parametrized by the inverse
Gauss map $$X: \mathbb{S}^n\rightarrow \mathcal{M}.$$
The support function $u: \mathbb{S}^n\rightarrow \mathbb{R}$ of $\mathcal{M}$ is defined
by
\begin{eqnarray*}\label{2.1}
u(x)=\sup\{\langle x, y\rangle: y \in \mathcal{M}\}.
\end{eqnarray*}
The supremum is attained at a point $y$ such that $x$ is the outer normal of $\mathcal{M}$
at $X$. It is easy to check that
\begin{eqnarray*}\label{2.2}
X=u(x)x+Du(x),
\end{eqnarray*}
where $D$ is the covariant derivative with respect to the standard metric $\sigma_{ij}$ of
the sphere $\mathbb{S}^n$. Hence
\begin{eqnarray}\label{2.3}
r=|X|=\sqrt{u^2+|Du|^2}.
\end{eqnarray}
Thus,
\begin{eqnarray}\label{2.3-1}
u=\frac{r^2}{\sqrt{r^2+|Dr|^2}}.
\end{eqnarray}
The second fundamental form of $\mathcal{M}$ is given by, see e.g.
\cite{An, Ur-1},
\begin{eqnarray}\label{2.4}
h_{ij}=u_{ij}+\sigma_{ij},
\end{eqnarray}
where $u_{ij}=D_iD_{j}u$ denotes the second order covariant derivative of $u$ with respect to the
spherical metric $\sigma_{ij}$. By Weingarten's formula,
\begin{eqnarray}\label{2.5}
\sigma_{ij}=\langle \frac{\partial \nu}{\partial x^i}, \frac{\partial \nu}{\partial x^j}\rangle
=h_{ik}g^{kl}h_{jl},
\end{eqnarray}
where $g_{ij}$ is the metric of $\mathcal{M}$ and $g^{ij}$ is its inverse. It follows from
\eqref{2.4} and \eqref{2.5} that the
principal radii of curvature of $\mathcal{M}$, under a smooth local orthonormal frame on $\mathbb{S}^n$, are
the eigenvalues of the matrix
\begin{eqnarray*}\label{2.6}
b_{ij}=u_{ij}+u\delta_{ij}.
\end{eqnarray*}
In particular, the Gauss curvature is given by
\begin{eqnarray*}
K=\frac{1}{\det(u_{ij}+u\delta_{ij})}.
\end{eqnarray*}

\subsection{Geometric flow and its associated functional}

\

For reader' convenience, the associated Mong-Amp\`ere
equation \eqref{Min} is restated here,
\begin{eqnarray*}
\frac{u\varphi(r)}{r^{n+1}}\cdot \det (u_{ij}+ u \ \delta_{ij})=f(x) \quad
\mbox{on} \quad \mathbb{S}^n.
\end{eqnarray*}
Recall the normalised anisotropic Gauss curvature flow \eqref{N-Eq}
\begin{equation*}
\left\{
\begin{aligned}
&\frac{\partial X}{\partial t}=-\theta(t)f(\nu)
\frac{r^{n+1}}{\varphi(r)}K\nu+X,&\\
&X(\cdot,0)=X_{0},
\end{aligned}
\right.
\end{equation*}
where $$\theta(t)=\int_{\mathbb{S}^n}\varphi(r(\xi, t))d\xi \
\bigg[\int_{\mathbb{S}^n}f(x)dx\bigg]^{-1}.$$ By the definition of
support function, we know $u(x, t)=\langle x, X(x, t)\rangle$.
Hence,
\begin{equation}\label{u-NEq}
\left\{
\begin{aligned}
&\frac{\partial u}{\partial t}(x, t)=-\theta(t)\frac{f(x)r^{n+1}}{\varphi(r)}K+u(x, t),&\\
&u(\cdot,0)=u_{0} .
\end{aligned}
\right.
\end{equation}
The normalised flow \eqref{N-Eq} can be also described by the following scalar equation for $r(\cdot, t)$
\begin{equation}\label{r-NEq}
\left\{
\begin{aligned}
&\frac{\partial r}{\partial t}(\xi, t)
=-\theta(t)\frac{f(x)r^{n+2}}{\varphi(r)u}K+r(\xi, t),&\\
&r(\cdot,0)=r_{0},
\end{aligned}
\right.
\end{equation}
in view of
\begin{eqnarray*}
\frac{1}{r(\xi, t)}\frac{\partial r(\xi, t)}{\partial t}
=\frac{1}{u(x, t)}\frac{\partial u(x, t)}{\partial t},
\end{eqnarray*}
see Section 3 in \cite{Chen-L19} for the proof.

For a convex body $\Omega\subset \mathbb{R}^{n+1}$, we define
\begin{eqnarray*}
V_\varphi(\Omega)=\int_{\mathbb{S}^n}d\xi\int_{0}^{r(\xi
,t)}\frac{\varphi(s)}{s}d s.
\end{eqnarray*}
When $\varphi(s)=s^q$, $V_\varphi(\Omega)$ be the $q$-volume of the
convex body $\Omega\subset \mathbb{R}^{n+1}$, see \cite{Chen-H19,
Chen-L19}. We show below that $V_\varphi(\Omega_t)$ is unchanged
under the flow \eqref{N-Eq}, where $\Omega_t$ is a compact convex
body in $\mathbb{R}^{n+1}$ with the boundary $\mathcal{M}_t$.

\begin{lemma}\label{V_q}
Let $X(\cdot, t)$ be a strictly convex solution to the flow
\eqref{N-Eq}, then we obtain
\begin{eqnarray*}
V_\varphi(\Omega_t)=V_\varphi(\Omega_0).
\end{eqnarray*}
\end{lemma}

\begin{proof}
\begin{eqnarray*}
&&\frac{d}{dt}V_{\varphi}(\Omega_t)
=\int_{\mathbb{S}^n}\frac{\varphi(r)}{r}\frac{\partial r}{\partial
t} d\xi\\&=&
\int_{\mathbb{S}^n}\frac{\varphi(r)}{r}
\bigg(-\theta(t)\frac{f(x)r^{n+2}}{\varphi(r)u}K+r(\xi,
t))\bigg)d\xi
\\&=&-\theta(t)\int_{\mathbb{S}^n}\frac{f(x)r^{n+1}}{u}Kd\xi+\int_{\mathbb{S}^n}\varphi(r)d\xi\\&=&0,
\end{eqnarray*}
where we use
\begin{eqnarray*}
\frac{dx}{d\xi}=\frac{r^{n+1}K}{u},
\end{eqnarray*}
see e.g. \cite{Chen-L19,Huang16}.
\end{proof}

Next, we define the functional
\begin{eqnarray*}
\mathcal{J}_{\varphi}(X(\cdot, t))=\int_{\mathbb{S}^n}\log u(x,
t)f(x)dx.
\end{eqnarray*}
The following lemma shows that the functional $\mathcal{J}_{\varphi}$ is non-increasing along the flow \eqref{N-Eq}.

\begin{lemma}\label{J}
Let $X(\cdot, t)$ be a strictly convex solution to the flow
\eqref{N-Eq}. For any $\varphi\geq 0$, the functional is
non-increasing along the flow \eqref{N-Eq}. In particular,
\begin{eqnarray*}
\frac{d}{dt}\mathcal{J}_{\varphi}(X(\cdot, t))\leq 0.
\end{eqnarray*}
and the equality holds if and only if $X_t$ satisfies the elliptic equation \eqref{Min} with $f$ replaced by $\theta(t) f$.
\end{lemma}

\begin{proof}
\begin{eqnarray*}
&&\frac{d}{dt}\mathcal{J}_{\varphi}(X(\cdot, t))\\&=&
\int_{\mathbb{S}^n}\frac{1}{u}\frac{\partial u(x, t)}{\partial
t}f(x)dx\\&=& \int_{\mathbb{S}^n}\frac{1}{u}
\bigg(-\theta(t)\frac{f(x)r^{n+1}}{\varphi(r)}K+u(x, t)\bigg)
f(x)dx\\&=&\bigg[\int_{\mathbb{S}^n}f(x)dx\bigg]^{-1}
\bigg\{-\int_{\mathbb{S}^n}\frac{u\varphi(r)}
{r^{n+1}K}dx\int_{\mathbb{S}^n}\frac{r^{n+1}K}{u\varphi(r)}f^2dx+
\int_{\mathbb{S}^n}fdx\int_{\mathbb{S}^n}
fdx\bigg\}\\&=&\bigg[\int_{\mathbb{S}^n}f(x)dx\bigg]^{-1}
\bigg\{-\int_{\mathbb{S}^n}\frac{u\varphi(r)}
{fr^{n+1}K}d\sigma\int_{\mathbb{S}^n}\frac{r^{n+1}K}{u\varphi(r)}fd\sigma+
\int_{\mathbb{S}^n}d\sigma\int_{\mathbb{S}^n}
d\sigma\bigg\}\\&\leq&0
\end{eqnarray*}
in view of
\begin{eqnarray*}
\int_{\mathbb{S}^n}d\sigma\int_{\mathbb{S}^n}
d\sigma\leq\int_{\mathbb{S}^n}\frac{u\varphi(r)}
{fr^{n+1}K}d\sigma\int_{\mathbb{S}^n}\frac{r^{n+1}K}{u\varphi(r)}fd\sigma
,
\end{eqnarray*}
which is implies by H$\ddot{o}$lder inequality, where
$d\sigma=f(x)dx$. Clearly,
the equality holds if and only if
\begin{eqnarray*}
\frac{f(x)r^{n+1}K}{u\varphi(r)}=\frac{1}{c(t)}.
\end{eqnarray*}
In this case, clearly, we have $\theta(t)=c(t)$. Thus,
$X(\cdot, t)$ satisfies the elliptic equation \eqref{Min} with $f$ replaced by $\theta(t) f$.
\end{proof}

Before closing this section, we prove the following basic properties
for any given $\Omega \in \mathcal{K}_0$, while smoothness of
$\partial \Omega$ is not required. First, we introduce the following
Lemma for convex bodies, see Lemma 2.6 in \cite{Chen-L19} for the
details.

\begin{lemma}\label{u-r}
Let $\Omega \in \mathcal{K}_0$. Let $u$ and $r$ be the support
function and radial function of $\Omega$, and $x_{\max}$ and
$\xi_{\min}$ be two points such that $u(x_{\max})=\max_{\mathbb{S}^n}u$
and $r(\xi_{\min})=\min_{\mathbb{S}^n}r$. Then
\begin{eqnarray*}
\max_{\mathbb{S}^n}u=\max_{\mathbb{S}^n}r \quad \mbox{and} \quad
\min_{\mathbb{S}^n}u=\min_{\mathbb{S}^n}r,
\end{eqnarray*}
\begin{eqnarray*}
u(x)\geq x\cdot x_{\max}u(x_{\max}), \quad \forall x \in
\mathbb{S}^n,
\end{eqnarray*}
\begin{eqnarray*}
r(\xi)\xi\cdot\xi_{\min}\geq r(\xi_{\min}), \quad \forall \xi \in
\mathbb{S}^n.
\end{eqnarray*}
\end{lemma}

Let $\mathcal{K}^{n+1}$=$\{K|K$ is convex body in
$\mathbb{R}^{n+1}\}$. Then, we have the following theorem (see also
\cite{Sch13}).
\begin{theorem}\label{K-0}
If $K_i\in\mathcal{K}^{n+1}$ and there exists a constant $ R>0$ such
that  $K_i\subset B_R$, then there exists a subsequence $ K_{i_j}$
and $K_0\in\mathcal{K}^{n+1}$ such that
$$ K_{i_j}\rightarrow K_0 ~~~~~\mbox{in the Hausdorff metric}. $$
\end{theorem}

To statement the following theorem, we first recall the definition
of the radial function of a convex body. (see also \cite{Sch13}).
\begin{defi}
Let $K\in\mathcal{K}^{n+1}$, $0\in K$, a radial function $r_K :
\mathbb{R}^{n+1}\backslash \{ 0 \} \rightarrow \mathbb{R}$ is
defined as
\begin{equation*}\label{1.5}
r_K (x)=\max \{r\geq0 | rx\in K\}.
\end{equation*}
\end{defi}

Now, the convergence of convex bodies imply the convergence of the
corresponding radial functions.
\begin{theorem}\label{r}
Let $K_0,K_i\in\mathcal{K}^{n+1}$, $0\in int K_0$ and
$K_i\rightarrow K_0$, then $r_{K_i}\rightrightarrows r_{K_0}$.
\end{theorem}
For the proof of the theorem above, see \cite{Sch13}.

\section{$C^0$, $C^1$-estimates}

In this section, we will derive the $C^0$, $C^1$-estimates of the
flow \eqref{N-Eq}. The key is the lower bound of $u$. The difficulty
of the proof lies the inhomogeneous term $\varphi(r)$.

\subsection{The upper bound of $u$ and gradient estimate}
It is easy to obtain the upper bound of $u$ and gradient estimate if
we notice that the functional $\mathcal{J}_{\varphi}$ is
non-increasing along the flow \eqref{N-Eq}, see Lemma \ref{J}.
\begin{lemma}\label{C0-u}
Let $X(\cdot, t)$ be a strictly convex solution to the flow
\eqref{N-Eq},  then we have
\begin{eqnarray}\label{C0}
u(\cdot, t)\leq C, \quad \forall t \in [0, T).
\end{eqnarray}
and
\begin{eqnarray}\label{C1}
|D u|(\cdot, t)\leq C, \quad \forall t \in [0, T).
\end{eqnarray}
\end{lemma}

\begin{proof}
Assume that $x_t$ is a point at where $u(\cdot, t)$ attains its spatial maximum, we know from Lemma \ref{J}
\begin{eqnarray*}
C\geq\int_{\mathbb{S}^n}\log u(x, t)f(x)dx\geq\int_{\{x \in
\mathbb{S}^n: x\cdot x_t>0\}}\log [x\cdot x_tu(x_t, t)]f(x)dx,
\end{eqnarray*}
which implies
\begin{eqnarray*}
C\geq \max_{\mathbb{S}^n}u(\cdot, t).
\end{eqnarray*}
This yields the inequality \eqref{C0}. Since
\begin{eqnarray*}
\max_{\mathbb{S}^n}|D u|(\cdot, t)\leq \max_{\mathbb{S}^n}u(\cdot,
t),
\end{eqnarray*}
we obtain \eqref{C1}.
\end{proof}

\subsection{The lower bound of $u$}
We get the lower bound of $u$ by the following gradient estimate for
Case (i) in Theorem \ref{main} and the fact that $f$ and $u_0$ are
even functions for Case (ii) in Theorem \ref{main}.

\begin{lemma}\label{Gra}
Let $X(\cdot, t)$ be a strictly convex solution to the flow
\eqref{N-Eq}, if
\begin{eqnarray}\label{Gra-con}
\max_{s>0}s\varphi'(s)\varphi^{-1}(s)<0,
\end{eqnarray}
then
\begin{eqnarray}\label{C1-log}
\max_{\mathbb{S}^n}\frac{|Du|}{u}(\cdot, t)\leq C, \quad \forall t \in [0, T).
\end{eqnarray}
\end{lemma}

\begin{proof}
Let $z=\log u$, it is straightforward to see
\begin{eqnarray*}
\frac{\partial z}{\partial t}
=-\theta(t)f(x)\frac{(1+|Dz|^2)^{\frac{n+1}{2}}}{\varphi(e^{z}\sqrt{1+|Dz|^2})}\frac{1}{\det(z_{ij}+z_i
z_j+\delta_{ij})}+1=Q(D^2 z, D z, z)+1.
\end{eqnarray*}
Set $\psi=\frac{|Dz|^2}{2}$. By differentiating the $\psi$,we have
\begin{equation*}
\begin{aligned}
\frac{\partial \psi}{\partial t}
&=(\frac{\partial}{\partial t}z_m) z^m\\
&=(\dot{z})_m z^m\\
&=Q_m z^m.
\end{aligned}
\end{equation*}
Then,
\begin{eqnarray*}
\frac{\partial \psi}{\partial t}=Q^{ij}z_{ijm} z^m +Q^kz_{km}
z^m+(-e^z\sqrt{1+|D z|^2}\varphi^{\prime}\varphi^{-1}Q|D z|^2+\langle D \log f,Dz\rangle Q).
\end{eqnarray*}
where
\begin{eqnarray*}
Q^{ij}=\frac{\partial Q}{\partial w_{ij}}=-Qw^{ij}, \quad
Q^{k}=\frac{\partial Q}{\partial z_{k}}.
\end{eqnarray*}
Interchanging the covariant derivatives, we have
\begin{equation*}
\begin{aligned}
\psi_{ij}&=(z_{mi} z^m)_{j}\\&=z_{mij} z^m+z_{mi} z^{m}_{\ j}\\
&=z_{imj} z^m+z_{mi} z^{m}_{\ j}\\
&=z_{ijm} z^m+\sigma_{ij}|Dz|^2-z_iz_j+z_{mi} z^{m}_{\ j}
\end{aligned}
\end{equation*}
in view of \eqref{3C}. Thus, we have
\begin{equation}\label{grass}
\begin{aligned}
\frac{\partial \psi}{\partial t}=&Q^{ij}\psi_{ij}+Q^k\psi_k
-Q^{ij}(\delta_{ij}|D z|^2-z_i z_j)\\&-Q^{ij}z_{mi} z^{m}_{\
j}+(-e^z\sqrt{1+|Dz|^2}\varphi^{\prime}\varphi^{-1}|D z|^2+\langle D \log f,Dz\rangle)Q \\
\leq& Q^{ij}\psi_{ij}+Q^k\psi_k
-Q^{ij}(\delta_{ij}|D z|^2-z_i z_j)\\&-Q^{ij}z_{mi} z^{m}_{\
j}+(-e^z\sqrt{1+|Dz|^2}\varphi^{\prime}\varphi^{-1}|D z|-C)Q|Dz|.
\end{aligned}
\end{equation}

Since the matrix $Q^{ij}$ and $\delta_{ij}|D\varphi|^2-\varphi_i
\varphi_j$ are positive definite, the third and forth  terms in the
right of \eqref{grass} are non-positive. And noticing that the fifth
term in the right of \eqref{grass} is nonpositive if \eqref{Gra-con}
holds true and $|Dz|\geq -\frac{C}{\max_{s>0}s\varphi'(s)\varphi^{-1}(s)}$. So we got the equation about $\psi$ as follows:
\begin{equation*}
\left\{
\begin{aligned}
&\frac{\partial \psi}{\partial t}\leq Q^{ij}\psi_{ij}+Q^k \psi_k
&&on~
\mathbb{S}^n\times(0,\infty),\\
&\psi(\cdot,0)=\frac{|Dz(\cdot,0)|^2}{2} &&on~\mathbb{S}^n.
\end{aligned}\right .\end{equation*}
Using the maximum principle, we get the gradient estimates of $z$.
\end{proof}

\begin{lemma}\label{C0-u-1}
Let $X(\cdot, t)$ be a strictly convex solution to the flow
\eqref{N-Eq},  then we have
\begin{eqnarray}\label{C00-1}
\frac{1}{C}\leq u(x, t)\leq C, \quad \forall (x, t) \in
\mathbb{S}^n\times[0, T).
\end{eqnarray}
if either (i) \eqref{Gra-con} holds true; or (ii) the assumption
\ref{Ass1} holds true, $f$ and $u_0$ are even functions.
\end{lemma}

\begin{proof}
We only need prove the first inequality in \eqref{C00-1} by noticing
Lemma \ref{C0-u}.

Case (i): If \eqref{Gra-con} holds true, we have by virtue of
\eqref{C1-log}
\begin{eqnarray*}
\max_{\mathbb{S}^n}\log u(\cdot, t)
-\min_{\mathbb{S}^n}\log u(\cdot, t)\leq C\max_{\mathbb{S}^n}\frac{|Du|}{u}(\cdot, t)\leq C,
\end{eqnarray*}
which implies the positive lower bound of $u$ together with
\eqref{C0}.

Case (ii): $f$ and $u_0$ are even. We have
\begin{equation}\label{ubl}
\int_{\mathbb{S}^n}d\xi
\int_{0}^{r(\xi,0)}\frac{\varphi(s)}{s}ds=\int_{\mathbb{S}^n}d\xi
\int_{0}^{r(\xi,t)}\frac{\varphi(s)}{s}ds
\end{equation}
by Lemma \ref{V_q}. Here we use the idea in \cite{Chen-H19} to
complete our proof by contradiction. Assume $r(\xi, t)$ is not
uniformly bounded away from $0$ which means there exists $\inf_{x
\in \mathbb{S}^n}r(\xi, t_i)\rightarrow 0$ as $i\rightarrow \infty$,
where $t_i \in [0, T)$. Since $f$ and $u_0$ are even, $r(\xi, t)$ is
even. Thus, $\Omega_t$ is a origin-symmetric body, where
$\Omega_{t}$ is the convex body containing the origin and $\partial
\Omega_t=\mathcal{M}_t$. Thus, using Theorem \ref{K-0}, we have
$\Omega_{t_i}$ (after choosing a subsequence) converges to a
origin-symmetric convex body $\Omega_0$. Then, we have by Theorem
\ref{r}
$$\inf_{\xi \in \mathbb{S}^n}r_{\Omega_0}(\xi)=0.$$
So, there exists $\xi_0 \in \mathbb{S}^n$ such that
$r_{\Omega_0}(\xi_0)=0$ and thus $r_{\Omega_0}(-\xi_0)=0$, which
implies $\Omega_0$ contained in a lower-dimensional subspace. This
means that
$$r(\xi, t_i)\rightarrow 0$$
as $i\rightarrow \infty$ almost everywhere with respect to the
spherical Lebesgue measure. Combined with bounded convergence
theorem, we conclude
\begin{equation*}
\int_{\mathbb{S}^n}d\xi
\int_{0}^{r(\xi,0)}\frac{\varphi(s)}{s}ds=\int_{\mathbb{S}^n}d\xi
\int_{0}^{r(\xi,t_i)}\frac{\varphi(s)}{s}ds \rightarrow 0
\end{equation*}
as $i\rightarrow \infty$, which is a contraction to \eqref{ubl}. So,
we complete our proof.
\end{proof}

The $C^0$ and $C^1$ estimates of $u$ imply the corresponding $C^0$
and $C^1$ estimates of $r$ by using \eqref{2.3-1} and Lemma
\ref{u-r}.

\begin{corollary}\label{C0-r}
Under the assumptions in Theorem \ref{main}, if $X(\cdot, t)$ is a
strictly convex solution to the flow \eqref{N-Eq},  then we have
\begin{eqnarray}\label{r-C0}
\frac{1}{C}\leq r(\xi, t)\leq C, \quad \forall (\xi, t) \in
\mathbb{S}^n\times[0, T),
\end{eqnarray}
\begin{eqnarray}\label{C1_r}
|D r|(\xi, t)\leq C, \quad \forall (\xi, t) \in
\mathbb{S}^n\times[0, T),
\end{eqnarray}
and
\begin{eqnarray}\label{scal_1}
\frac{1}{C}\leq \theta(t)\leq C, \quad \forall t \in [0, T).
\end{eqnarray}
\end{corollary}

\section{$C^2$-estimates}

In this section we establish uniformly positive and lower bounds for
the principle curvatures for the normalised flow \eqref{N-Eq}. We
first use the technique that was first introduced by Tso \cite{Tso}
to derive the upper bound of the Gauss curvature along the flow
\eqref{N-Eq}, see also the proof of Lemma 4.1 in \cite{LiQ} and
Lemma 5.1 in \cite{Chen-H19}.

\begin{lemma}\label{K}
Let $X(\cdot, t)$ be a strictly convex solution to the flow
\eqref{N-Eq} which encloses the origin for $t \in [0,T) $. Then,
there exists a positive constant $C$ depending only $\varphi$,
$\max_{\mathbb{S}^n\times[0, T)}u$ and $\min_{\mathbb{S}^n\times[0,
T)}u$, such that
\begin{eqnarray*}
\max_{\mathbb{S}^n}K(\cdot, t)\leq C, \quad \forall t \in[0, T).
\end{eqnarray*}
\end{lemma}

\begin{proof}
We apply the maximum principle to the following auxiliary function
defined on the unit sphere $\mathbb{S}^n$
$$W(x, t)=\frac{1}{\theta(t)}\frac{-u_t+u}{u-\varepsilon_0}=\frac{f(x)}{\varphi(r)}r^{n+1}\frac{K}{u-\varepsilon_0},$$
where
$$\varepsilon_0=\frac{1}{2}\min_{(x, t) \in \mathbb{S}^n\times [0, T)} u(x, t)>0.$$
At the maximum $x_0$ of $W$ for any fixed $t \in [0, T)$, we have at
$(x_0, t)$
\begin{eqnarray}\label{K-1}
0=\theta(t)W_i=\frac{-u_{ti}+u_i}{u-\varepsilon_0}+\frac{u_t-
u}{(u-\varepsilon_0)^2}u_i,
\end{eqnarray}
and
\begin{eqnarray}\label{K-2}
0\geq \theta(t)D_{ij}^2
W=\frac{-u_{tij}+u_{ij}}{u-\varepsilon_0}+\frac{(u_t-u)
u_{ij}}{(u-\varepsilon_0)^2},
\end{eqnarray}
where \eqref{K-1} was used in deriving the second equality above.
The inequality \eqref{K-2} should be understood in sense of
positive-semidefinite matrix. Hence,
\begin{eqnarray*}
u_{tij}+u_{t}\delta_{ij}\geq
\theta(t)(-b_{ij}+\varepsilon_0\delta_{ij})W+b_{ij}.
\end{eqnarray*}
Thus,
\begin{eqnarray*}
K_t=-Kb^{ij}(u_{tij}+u_{t}\delta_{ij})\leq-n K-\theta(t)K
W(-n+\varepsilon_0 H),
\end{eqnarray*}
where $H$ denotes the mean curvature of $X(\cdot, t)$. Noticing that
$H\geq n K^{\frac{1}{n}}$, we obtain
\begin{eqnarray*}
K_t\leq C W(1+W)-CW^{2+\frac{1}{n}}.
\end{eqnarray*}
Using the equation \eqref{u-NEq} and the inequality above, we have
\begin{eqnarray*}\label{K-t}
W_t&=&\bigg[\frac{f(x)}{\varphi(r)}\frac{r^{n+1}}{u-\varepsilon_0}\bigg]_tK
+\bigg[\frac{f(x)}{\varphi(r)}\frac{r^{n+1}}{u-\varepsilon_0}\bigg]K_t
\\&\leq&CW^2+C W-CW^{2+\frac{1}{n}},
\end{eqnarray*}
in view of
\begin{eqnarray*}\label{K-3}
u_{t} \approx CW+C, \quad  r_{t}=\frac{u u_t+u^k u_{kt}}{r}\approx C
W+C.
\end{eqnarray*}
Without loss of generality we assume that $K\approx W\gg 1$, which
implies that
\begin{eqnarray*}
W_t\leq 0.
\end{eqnarray*}
Therefore, we arrive at $W\leq C$ for some constant $C>0$ depending
on the $C^1$-norm of $r$ and $\varepsilon_0$. Thus, the priori bound
follows consequently.
\end{proof}

Now, we show the principle curvatures of $X(\cdot, t)$ are bounded
from below along the flow \eqref{N-Eq}. The proof is similar to
Lemma 4.2 in \cite{LiQ} and Lemma 5.1 in \cite{Chen-H19}.

\begin{lemma}\label{C22}
Let $X(\cdot, t)$ be a strictly convex solution to the flow
\eqref{N-Eq} which encloses the origin for $t \in [0,T) $. Then,
there exists a positive constant $C$ depending only $\varphi$, $q$,
$\max_{\mathbb{S}^n\times[0, T)}u$ and $\min_{\mathbb{S}^n\times[0,
T)}u$, such that the principle curvatures of $X(\cdot, t)$ are
bounded from below
\begin{eqnarray}\label{scal_2}
\kappa_i(x, t)\geq C, \quad \forall (x, t) \in \mathbb{S}^n\times[0,
T), \ \mbox{and}\ i=1,2...,n.
\end{eqnarray}
\end{lemma}

\begin{proof}
We consider the auxiliary function
\begin{eqnarray*}
\widetilde{\Lambda}(x, t)=\log \lambda_{\max}(\{b_{ij}\})-A\log u+B
|D u|^2,
\end{eqnarray*}
where $A$ and $B$ are positive constants which will be chosen later,
and $\lambda_{max}(\{b_{ij}\})$ denotes the maximal eigenvalue of
$\{b_{ij}\}$. For convenience, we write $\{b^{ij}\}$ for
$\{b_{ij}\}^{-1}$.

For any fixed $t \in [0, T)$, we assume the maximum
$\widetilde{\Lambda}$ is achieved at some point $x_0 \in
\mathbb{S}^n$. By rotation, we may assume $\{b^{ij}(x_0, t)\}$ is
diagonal and $\lambda_{\max}(\{b_{ij}\})(x_0, t)=b_{11}(x_0, t)$.
Thus, it is sufficient to prove $b_{11}(x_0, t)\leq C$.

Then, we define a new auxiliary function
\begin{eqnarray*}
\Lambda(x, t)=\log b_{11}-A\log u+B |D u|^2,
\end{eqnarray*}
which attains the local maximum at $x_0$ for fixed time $t$. Thus,
we have at $x_0$
\begin{eqnarray}\label{C2-1d}
0=D_i\Lambda=b^{11}b_{11; i}-A\frac{u_i}{u}+2B \sum_{k}u_k u_{ki}
\end{eqnarray}
and
\begin{eqnarray}\label{C2-2d}
0\geq D_iD_j\Lambda=b^{11}b_{11; ij}-(b^{11})^2 b_{11; i}b_{11;
j}-A\bigg(\frac{u_{ij}}{u}-\frac{u_i u_j}{u^2}\bigg) +2B
\sum_{k}\bigg(u_{kj} u_{ki}+u_ku_{kij}\bigg).
\end{eqnarray}
We can rewrite the equation \eqref{u-NEq} as
\begin{eqnarray}\label{C2-1}
\log(u-u_t)=-\log\det(b)+\alpha(x, t),
\end{eqnarray}
where
\begin{eqnarray*}
\alpha(x, t)=\log
\bigg(\theta(t)\frac{f(x)r^{n+1}}{\varphi(r)}\bigg).
\end{eqnarray*}
Differentiating \eqref{C2-1} gives
\begin{eqnarray}\label{C2-00}
\frac{u_k-u_{kt}}{u-u_t}=-b^{ij}b_{ij; k}+D_k \alpha
\end{eqnarray}
and
\begin{eqnarray}\label{C2-11}
\frac{u_{11}-u_{11t}}{u-u_t}=\frac{(u_{1}-u_{1t})^2}{(u-u_t)^2}-b^{ii}b_{ii;
11} +b^{ii}b^{jj}(b_{ij; 1})^2+D_1D_1 \alpha.
\end{eqnarray}
Recalling the Ricci identity \eqref{RI}
\begin{eqnarray*}
b_{ii; 11}=b_{11; ii}-b_{11}+b_{ii},
\end{eqnarray*}
which is taken into \eqref{C2-11} implies
\begin{eqnarray}\label{C2-111}
\frac{u_{11}-u_{11t}}{u-u_t}=\frac{(u_{1}-u_{1t})^2}{(u-u_t)^2}
-b^{ii}b_{11; ii}+\sum_{i}b^{ii}b_{11}-n +b^{ii}b^{jj}(b_{ij;
1})^2+D_1D_1 \alpha.
\end{eqnarray}
So, we have
\begin{eqnarray}\label{2-t}
\frac{\partial_t \Lambda}{u-u_t}&=&
b^{11}\bigg(\frac{u_{11t}-u_{11}}{u-u_t}+\frac{u_{11}+u-u+u_t}{u-u_t}\bigg)-
A\frac{1}{u}\frac{u_t-u+u}{u-u_t}+2B\frac{u^k u_{kt}}{u-u_t}
\\ \nonumber&=&b^{11}\bigg[-\frac{(u_{1}-u_{1t})^2}{(u-u_t)^2}
+b^{ii}b_{11; ii}-\sum_{i}b^{ii}b_{11} -b^{ii}b^{jj}(b_{ij;
1})^2-D_1D_1 \alpha\bigg]\\ \nonumber&&+\frac{1-A}{u-u_t}
+\frac{A}{u}+2B\frac{\sum_{k}u_k u_{kt}}{u-u_t}+(n-1)b^{11}.
\end{eqnarray}
We know from \eqref{C2-2d} and \eqref{C2-00}
\begin{eqnarray*}
0&\geq&b^{11}[b^{ii}b_{11; ii}-b^{ii}b^{11} (b_{i1;
1})^2]-A\frac{n}{u}+A\sum_{i}b^{ii}+Ab^{ii}\frac{u_i
u_i}{u^2}\\&&+2B \bigg[b^{ii}(b_{ii}-u)^2+\sum_{k}u_k(D_k
\alpha-\frac{u_k-u_{kt}}{u-u_t})-b^{ii}u_iu_i\bigg]\\&\geq&b^{11}[b^{ii}b_{11;
ii}-b^{ii}b^{jj} (b_{ij; 1})^2]-A\frac{n}{u}+A
\sum_{i}b^{ii}+Ab^{ii}\frac{u_i u_i}{u^2}\\&&+2B
\bigg[\sum_{i}b^{ii}(b^{2}_{ii}-2ub_{ii})+\sum_{k}u_k(D_k
\alpha-\frac{u_k-u_{kt}}{u-u_t})-b^{ii}u_iu_i\bigg]\\&\geq&b^{11}[b^{ii}b_{11;
ii}-b^{ii}b^{jj} (b_{ij; 1})^2]-A\frac{n}{u}+A
\sum_{i}b^{ii}+Ab^{ii}\frac{u_i u_i}{u^2}\\&&+2B
\bigg[\sum_{i}b_{ii}-2nu+\sum_{k}u_k(D_k
\alpha-\frac{u_k-u_{kt}}{u-u_t})-b^{ii}u_iu_i\bigg].
\end{eqnarray*}
Thus, plugging the inequality above into \eqref{2-t} gives
\begin{eqnarray}\label{2-t-1}
\frac{\partial_t \Lambda}{u-u_t}&\leq&-b^{11}D_1D_1
\alpha-2B\sum_{k}u_kD_k \alpha+\frac{1-A+2B|Du|^2}{u-u_t}
\\ \nonumber&&+\frac{(n+1)A}{u}+(n-1)b^{11}+(2B|Du|-A-1)\sum_{i}b^{ii}\\ \nonumber&&-Ab^{ii}\frac{u_i
u_i}{u^2}-2B \sum_{i}b_{ii}+4nBu.
\end{eqnarray}
Now, we need estimate the first two terms in the inequality above.
Clearly, a direct calculation results in
\begin{eqnarray*}
r_i=\frac{u u_i+\sum_{k}u_k u_{ki}}{r}=\frac{u_i b_{ii}}{r}
\end{eqnarray*}
and
\begin{eqnarray*}
r_{ij}=\frac{u u_{ij}+u_iu_j+\sum_{k}u_k u_{kij}+\sum_{k}u_{kj}
u_{ki}}{r}-\frac{u_i u_i b_{ii}b_{jj}}{r^3}.
\end{eqnarray*}
Hence, we obtain by Lemma \ref{C0-u}, Lemma \ref{C0-u-1} and
Corollary \ref{C0-r}
\begin{eqnarray*}
&&-b^{11}D_1D_1\alpha-2B\sum_{k}u_kD_k \alpha\\&=&-
b^{11}\bigg[\frac{f_{11}}{f}-\frac{f^{2}_{1}}{f^2}-
(n+1)\frac{r_{1}^{2}}{r^2}+\frac{(\varphi^{\prime})^2r_{1}^{2}}{\varphi^2}-\frac{\varphi^{\prime
\prime}r_{1}^{2}}{\varphi}\bigg]
-b^{11}\bigg[(n+1)\frac{1}{r}-\frac{\varphi^{\prime}}{\varphi}\bigg]r_{11}\\&&-2B\sum_{k}u_k
\bigg(\frac{f_k}{f}+[(n+1)\frac{1}{r}-\frac{\varphi^{\prime}}{\varphi}]r_k\bigg)\\&\leq&C
b^{11}(1+b_{11})+C
B-\bigg[(n+1)\frac{1}{r}-\frac{\varphi^{\prime}}{\varphi}\bigg](b^{11}r_{11}+2B
u_k r_k)\\&\leq&C b^{11}(1+b_{11}+b^{2}_{11})+C
B-\bigg[(n+1)\frac{1}{r}-\frac{\varphi^{\prime}}{\varphi}\bigg]\bigg(b^{11}\frac{u_ku_{k11}}{r}+2B
\frac{u_k u_ku_{kk}}{r}\bigg).
\end{eqnarray*}
Then, using \eqref{C2-1d}, we have
\begin{eqnarray*}
&&-b^{11}D_1D_1\alpha-2B\sum_{k}u_kD_k \alpha
\\&\leq&C b^{11}(1+b_{11}+b^{2}_{11})+C
B-\bigg[(n+1)\frac{1}{r}-\frac{\varphi^{\prime}}{\varphi}\bigg]\frac{u_k}{r}\bigg(A\frac{u_k}{u}-b^{11}u_1\delta_{k1}\bigg)
\\&\leq&C b^{11}(1+b_{11}+b^{2}_{11})+C
B+CA.
\end{eqnarray*}
Thus, using the inequality above, we conclude from \eqref{2-t-1}
\begin{eqnarray*}
\frac{\partial_t \Lambda}{u-u_t}&\leq&C(b^{11}+1+b_{11})+C B+CA
+\frac{(n+1)A}{u}+(n-1)b^{11}+(2B|Du|-A-1)\sum_{i}b^{ii}\\&&-Ab^{ii}\frac{u_i
u_i}{u^2}-2B \sum_{i}b_{ii}+4nBu\\&<&0,
\end{eqnarray*}
provided $b_{11}\gg 1$ and if we choose $A\gg B$. So we complete the
proof.
\end{proof}

\section{The convergence of the normalised flow}

With the help of a prior estimates in the section above, we show the
long-time existence and asymptotic behaviour of the normalised flow
\eqref{N-Eq} which complete Theorem \ref{main}.

\begin{proof}
Since the equation \eqref{u-NEq} is parabolic, we have the short
time existence. Let $T$ be the maximal time such that $u(\cdot, t)$
is a positive, smooth and strictly convex solution to \eqref{u-NEq}
for all $t \in [0, T)$. Lemmas \ref{C0-u}, \ref{Gra}, \ref{K} and
Corollary \ref{C0-r} enable us to apply Lemma \ref{C22} to the
equation \eqref{u-NEq} and thus we can deduce a uniformly lower
estimate for the biggest eigenvalue of $\{(u_{ij}+u\delta_{ij})(x,
t)\}$. This together with Lemma \ref{C22} implies
\begin{eqnarray*}
C^{-1}I \leq (u_{ij}+u\delta_{ij})(x, t)\leq C I, \quad \forall (x,
t) \in \mathbb{S}^n \times [0, T),
\end{eqnarray*}
where $C>0$ depends only on $n, \alpha, f$ and $u_0$. This shows
that the equation \eqref{u-NEq} is uniformly parabolic. Using
Evans-Krylov estimates and Schauder estimates, we obtain
\begin{eqnarray*}
|u|_{C^{l, m}_{x, t}(\mathbb{S}^n \times [0, T))}\leq C_{l, m}
\end{eqnarray*}
for some $C_{l, m}$ independent of $T$. Hence $T=\infty$. The
uniqueness of the smooth solution $u(\cdot, t)$ follows by the
parabolic comparison principle.

By the monotonicity of $\mathcal{J}_{\varphi}$ (See Lemma
\ref{J}), and noticing that
\begin{eqnarray*}
|\mathcal{J}_{\varphi}(X(\cdot, t))|\leq C,\quad \forall t \in
[0, \infty),
\end{eqnarray*}
we conclude that
\begin{eqnarray*}
\int_{0}^{\infty}|\frac{d}{dt}\mathcal{J}_{\varphi}(X(\cdot,
t))|\leq C.
\end{eqnarray*}
Hence, there is a sequence $t_i\rightarrow\infty$ such that
\begin{eqnarray*}
\frac{d}{dt}\mathcal{J}_{\varphi}(X(\cdot, t_i))\rightarrow 0.
\end{eqnarray*}
In view of Lemma \ref{J}, we see that $u(\cdot, t_i)$ converges
smoothly to a positive, smooth and strictly convex $u_{\infty}$
solving \eqref{Min} with $f$ replaced by $\lambda_0f$ with
$\lambda_0=\lim_{t_i\rightarrow \infty}\theta(t_i)$.
\end{proof}

\end{document}